\documentclass[12pt]{amsart}
\usepackage{amscd,amsmath,amssymb,amsfonts}
\usepackage{dsfont}
\usepackage[cmtip, all]{xy}

\usepackage{graphicx,amssymb,amsmath,amsfonts,amsthm,amscd,hyperref,textcomp,relsize,colortbl}
\usepackage{tikz-cd}
\usetikzlibrary{babel}

\theoremstyle{plain}
\newtheorem{thm}{Theorem}
\newtheorem{lem}[thm]{Lemma}
\newtheorem{cor}[thm]{Corollary}
\newtheorem{prop}[thm]{Proposition}

\theoremstyle{definition}

\newtheorem{rmk}[thm]{Remark}

\numberwithin{thm}{section}
\numberwithin{equation}{section}

\newcommand{\Gal}{{\rm Gal}}

\newcommand{\Trace}{{\rm Trace}}

\newcommand{\sF}{{\mathcal F}}
\newcommand{\sG}{{\mathcal G}}
\newcommand{\sH}{{\mathcal H}}

\newcommand{\sK}{{\mathcal K}}
\newcommand{\sL}{{\mathcal L}}

\newcommand{\A}{{\mathbb A}}

\newcommand{\C}{{\mathbb C}}

\newcommand{\F}{{\mathbb F}}
\newcommand{\G}{{\mathbb G}}

\newcommand{\Q}{{\mathbb Q}}

\newcommand{\Z}{{\mathbb Z}}

\newcommand{\triv}{{\mathds{1}}}

\newcommand{\FF}{\mathcal F}

\newcommand{\NN}{\mathbb N}

\begin{document}
\title{A rigid local system with monodromy group $2.J_2$}
\author{Nicholas M. Katz and Antonio Rojas-Le\'{o}n}
\address{Princeton University, Mathematics, Fine Hall, NJ 08544-1000, USA }
\email{nmk@math.princeton.edu}
\address{Departamento de \'{A}lgebra, Universidad de Sevilla, c/Tarfia, s/n, 41012 Sevilla - SPAIN}
\email{arojas@us.es}
\thanks{The second author was partially supported by MTM2016-75027-P (Ministerio de Econom\'{\i}a y Competitividad) and FEDER}

\maketitle

\begin{abstract} 
We exhibit a rigid local system of rank six on the affine line in characteristic $p=5$ whose arithmetic and geometric monodromy groups are
the finite group $2.J_2$ ($J_2$ the Hall-Janko sporadic group) in one of its two (Galois-conjugate) irreducible representation of degree six.
\end{abstract}

\tableofcontents

\section{Introduction: the general setting}
We fix a prime number $p$, a prime number $\ell \neq p$, and a nontrivial $\overline{\Q_\ell}^\times$-valued additive character $\psi$ of $\F_p$. For $k/\F_p$ a finite extension, we denote by  $\psi_k$ the nontrivial additive character of $k$ given by $\psi_k :=\psi\circ \Trace_{k/\F_p}$. In perhaps more down to earth terms, we fix a nontrivial $\Q(\mu_p)^\times$-valued additive character $\psi$ of $\F_p$, and a field embedding of
$\Q(\mu_p)$ into  $\overline{\Q_\ell}$ for some $\ell \neq p$.

Given an integer $D \ge 3$ which is prime to $p$, we form the local system $\sF_{p,D}$ on $\A^1/\F_p$ whose trace function, at $k$-valued points $t \in \A^1(k) = k$, is given by
$$t \mapsto -\sum_{x \in k}\psi_k(x^D + tx).$$
This is a geometrically irreducible rigid local system, being the Fourier Transform of the rank one local system $\sL_{\psi(x^D)}$. It has rank $D-1$, and each of its $D-1$ $I(\infty)$-slopes is $D/(D-1)$. It is pure of weight one.  [It is the local system  $\sF(\F_p,\psi, \triv, D)$ of \cite{Ka-RLSA}.]

Let us further fix a choice of $\sqrt{p} \in \overline{\Q_\ell}^\times$. For each finite extension $k/\F_p$, we then use this choice of$\sqrt{p}$ to define $\sqrt{\#k} := \sqrt{p}^{\deg(k/\F_p)}$. We then define the ``half"-Tate twisted local system
$$\sG_{p,D} := \sF_{p,D}(1/2)$$
whose trace function, at  at $k$-valued points $t \in \A^1(k) = k$, is given by
$$t \mapsto -\sum_{x \in k}\psi_k(x^D + tx)/\sqrt{\#k}$$

The local system $\sG_{p,D}$ is pure of weight zero. 
\begin{lem} The determinant $\det(\sG_{p,D})$ is arithmetically of finite order. More precisely,  $\det(\sG_{p,D})^{\otimes 4p}$ is arithmetically trivial.
\end{lem}
\begin{proof}It suffices to show that after extension of scalars from $\F_p$ to its quadratic extension $\F_{p^2}$, the $2p$'th power is trivial, i.e.,
that if $k/\F_p$ is a finite extension of even degree $2d$, then the determinant takes values in $\mu_{2p}$.To see this, note that the twisting factor $\sqrt{p}^{2d}=p^d \in \Q$, so this determinant has values in $\Q(\mu_p)$ which are units at all finite places of residue characteristic not $p$ (use the $\ell$-adic incarnations) and which have absolute value $1$ at all archimedean places of $\Q(\mu_p)$. Because there is a unique $p$-adic place of $\Q(\mu_p)$, the product formula shows that the determinant has values which are also units at $p$, and hence are roots of unity in $\Q(\mu_p)$, i.e., they are $2p$'th roots of unity.
\end{proof}

When we view the local system $\sG_{p,D}$ as a representation
$$\rho_{\sG_{p,D}}: \pi_1(\A^1/\F_p) \rightarrow GL(D-1,\overline{\Q_\ell}),$$
the Zariski closure of the image of $\pi_1(\A^1/\F_p) $ is defined to be the arithmetic monodromy group $G_{arith}$. The Zariski closure of the image of the normal subgroup $\pi_1^{geom}:=\pi_1(\A^1/\overline{\F_p}) $ is defined to be the geometric monodromy group $G_{geom}$. 
Thus we have inclusions of algebraic groups over $\overline{\Q_\ell}$: $$G_{geom}\lhd G_{arith} \subset GL(D-1).$$

Applying \cite[8.14.5, $(1) \iff (2) \iff (6)$]{Ka-ESDE} in the particular case of $\sG_{p,D}$, we have
\begin{prop}\label{arithfin}The following conditions are equivalent.
\begin{itemize}
\item [(1)] $\sG_{p,D}$ has finite $G_{geom}$.
\item [(1bis)] $\sF_{p,D}$ has finite $G_{geom}$.
\item [(2)] $\sG_{p,D}$ has finite $G_{arith}$.
\item [(3)] All traces of $\sG_{p,D}$ are algebraic integers.
\end{itemize}
\end{prop}

When $D \ge 3$ is odd (and prime to $p$), the local system  $\sF_{p,D}$ is symplectically self-dual. As shown in \cite[Proposition 4 and Corollary 6]{R-L}, its $G_{geom}$ is either finite or it is the full symplectic group  ${\rm Sp}(D-1)$. When $D \ge 3$ is even (and prime to $p$),  the same reference shows that $G_{geom}$ is either finite or is $SL(D-1)$. The proof of \cite[Proposition 4 and Lemma 5]{R-L} also shows 
that when $D$ is not of the form $1+q$ for $q$ a power of $p$, then $\sF_{p,D}$  is not induced (i.e., the given representation of its 
$G_{geom}$ is not induced). Indeed, by the result  \cite[11.1]{Such} of  \v{S}uch, if the representation were induced, it would be Artin-Schreier induced, and that is what is ruled out when $D$ is not of the form $1+q$. [When $D=1+q$, then $G_{geom}$ is, by Pink \cite[20.3]{Ka-RLSA}, a finite $p$-group, and (hence) the representation is induced.]

When $D \ge 3$ is prime to $p$, the trace function of $\sG_{p,D} $ takes values in $\Z[\mu_p]$. If moreover $p$ is $1$ mod $4$, then
we can choose either quadratic Gauss sum, a quantity which itself lies in $\Z[\mu_p]$, as our $\sqrt{p}$, and hence all traces of $\sG_{p,D}$ lie in $\Z[\mu_p][1/p]$.  If $p$ is not $1$ mod $4$, this remains true for traces of the pullback of $\sG_{p,D}$ to $\A^1/\F_{p^2}$. In either case, the traces of $\sG_{p,D}$ in question are algebraic integers if and only if they all have
${\rm ord}_p \ge 0$.

\begin{rmk}\label{tracefld}When $D \ge 3$ is prime to $p$ and {\bf odd}, then the traces of $\sF_{p,D} $ lie in the real subfield $\Q(\mu_p)^+$.
If in addition $p$ is $1$ mod $4$, then either quadratic gauss sum is $\pm \sqrt{p}$ also lies in this field, and hence  $\sG_{p,D} $ has traces in 
 $\Q(\mu_p)^+$.
 \end{rmk}
 
Results of Kubert, explained in  \cite[4.1,4.2,4.3]{Ka-RLSA} and discovered independently in \cite[Cor. 4, Cor. 5]{R-L}, show that $G_{geom}$ and $G_{arith}$ for $\sG_{p,D}$ are finite when $q$ is a power of $p$ and $D$ is any of
$$q+1,\ \ \frac{q+1}{2} {\rm \ with \ odd \  }q,  \frac{q^n+1}{q+1} {\rm \ with \ odd \ }n.$$
Let us call these the Kubert cases.
In \cite[17.1, 17.2]{Ka-RLSA} and \cite[3.4]{Ka-Ti-RLSMFUG} their $G_{geom}$ groups are  determined for all odd $q$.

Both authors have given numerical criteria for $\sG_{p,D}$ to have finite $G_{geom}$ and $G_{arith}$, cf. \cite[first paragraph after 5.1]{Ka-RLSA} and \cite[Thm. 1]{R-L}. The second author did extensive computer experiments to find other $(p,D)$ than the Kubert cases for which $\sG_{p,D}$ seemed to have finite $G_{geom}$ (i.e., where many many traces were all algebraic integers). For primes $p \le 11$ and $D \le 10^6$, there was only one non-Kubert candidate, the case $p=5,\  D=7$.

In the first part of this paper, we prove that $\sF_{5,7}$ has finite $G_{geom}$ (and hence, by Proposition \ref{arithfin},  that $\sG_{5,7}$ has finite $G_{geom}$ and finite $G_{arith}$). In the second part, we show that 
 $G_{geom}=G_{arith}= 2.J_2$ in one of its two six-dimensional irreducible representations. These two representations are symplectic. Their character values lie in $\Z[\sqrt{5}]$ and are Galois-conjugates of each other. As Guralnick and Tiep point out \cite[Table 1]{G-T}, the group $2.J_2$, sitting inside $Sp(6,\C)$, has the exotic property that it has the same moments $M_n$ (dimension of the space of invariants in the $n$'th tensor power of the given six-dimensional representation) as the ambient group ${\rm Sp}(6,\C)$ for $M_1$ through $M_{11}$; one needs $M_{12}$ to distinguish them.
 
 It is not clear whether there ``should" be infinitely many $(p,D)$ other than the Kubert cases for which $\sG_{p,D}$ has finite $G_{arith}$, or finitely many, or just this one $(5,7)$ case. Much remains to be done.

\section{Finiteness of the monodromy}

In this section we will prove that the sheaf $\FF_{5,7}$ has finite geometric monodromy. We will do so by applying the numerical criterion proven in \cite[Theorem 1]{R-L}, which we recall here. For a prime $p$ and an integer $x \ge 0$, we define 
$$[x]_{p,\infty}:={\rm \ the\  sum \ of \ the\  digits \ of\  the \ }p{\rm -adic\  expansion\ of\ }x,$$ 
using the usual digits $\{0,1,2,...,p-1\}$.

For every $r\geq 1$ we define $[x]_{p,r}=[x]_{p,\infty}$ if $1\le x \le p^r-1$, and we extend the definition to every integer $x$ by imposing that $[x]_{p,r}=[y]_{p,r}$ if $x\equiv y$ (mod $p^r-1$).[Thus we are using $\{1,2,...,p^r-1\}$ as representatives of $\Z/(p^r - 1)\Z$.] 

From \cite[Thm. 1]{R-L}, we have

\begin{thm}\label{numcrit}
The sheaf $\FF_{p,d}$ has finite geometric monodromy if and only if the inequality
 $$
 [dx]_{p,r}\leq [x]_{p,r}+\frac{r(p-1)}{2}
 $$
holds for every $r\geq 1$ and every integer $0<x<p^r$.
\end{thm}

Let us enumerate some basic properties of the functions $[-]_{p,\infty}$ and $[-]_{p,r}$. 

\begin{prop}
 For strictly positive integers $x$ and $y$, and for $r\in\NN\cup\{\infty\}$, we have:
 \begin{enumerate}
 \item $[x+y]_{p,r}\leq [x]_{p,r}+[y]_{p,r}$
 \item $[x]_{p,r}\leq [x]_{p,\infty}$
 \item $[px]_{p,r}=[x]_{p,r}$
 \end{enumerate}
\end{prop}

\begin{proof}
 We first prove (1) for $r=\infty$. Note that $[x]_{p,\infty}$ is the minimal number of terms in any decomposition of $x$ as a sum of powers of $p$. By taking the sum of the $p$-adic expansions of $x$ and $y$ we see that $x+y$ can be written as a sum of $[x]_{p,\infty}+[y]_{p,\infty}$ powers of $p$, and the inequality follows.
 
 For (2) we proceed by induction on $x$: for $0<x<p^r$ it is obvious by definition. If $x \ge p^r$, let $s$ be the largest integer such that $p^s\leq x$. Then $s\geq r$ and $[x-p^s]_{p,\infty}=[x]_{p,\infty}-1$. Since $x\equiv x-p^s+p^{s-r}$ (mod $p^r-1$), while $x > x-p^s+p^{s-r} > 0$,
 we have, by induction on $x$ and (1)$_\infty$,
 $$
 [x]_{p,r}=[x-p^s+p^{s-r}]_{p,r}\leq[x-p^s+p^{s-r}]_{p,\infty}\leq[x-p^s]_{p,\infty}+[p^{s-r}]_{p,\infty}=
 $$
 $$
 =[x]_{p,\infty}-1+1=[x]_{p,\infty}.
 $$

 In order to prove (1) for finite $r$ we can assume that $x,y<p^r$. Then by (2) and (1)$_\infty$, we have
 $$
 [x+y]_{p,r}\leq [x+y]_{p,\infty}\leq [x]_{p,\infty}+[y]_{p,\infty}=[x]_{p,r}+[y]_{p,r}.
 $$
 
 Finally, (3) is obvious for $r=\infty$. For finite $r$, note that if $x=a_{r-1}p^{r-1}+\cdots+a_1p+a_0$ is the $p$-adic expansion of $x<p^r$, then $px\equiv a_{r-2}p^{r-1}+\cdots+a_1p^2+a_0p+a_{r-1}$ (mod $p^r-1$), so $[px]_{p,r}=[x]_{p,r}=a_{r-1}+\cdots+a_1+a_0$.
 
\end{proof}

We now fix $p=5$ and $d=7$.

\begin{lem}\label{l1}
 Let $r$ be a positive integer and $0 \le x<5^r$ an integer such that $x\not\equiv 2 \mod 5$. Then $[7x]_{5,\infty}\leq[x]_{5,\infty}+2r$.
\end{lem}

\begin{proof} We proceed by induction on $r$. For $r=1$ and $r= 2$ one checks it by hand.

Now let $r\geq 3$ and $0 \le x<5^r$ with $x\not\equiv 2 \mod 5$. 
If $0 \le x<5^{r-1}$ the stronger inequality $[7x]_\infty\leq [x]_\infty+2r-2$
holds by induction, so we may assume that $5^{r-1}\leq x<5^r$. Consider the $5$-adic expansion of $x$, which has $r$ digits, the last one being $\neq 2$ by hypothesis. We distinguish two cases:

\emph{Case 1:} The constant term is not $2$, and there is some other digit $\neq 2$, say the one multiplying $5^s$, for some $s$ with $r >s >0$. Write $x=5^sy+z$, with $0 \le z<5^s$, $0 \le y<5^{r-s}$ (that is, split the first $r-s$ and the last $s$ $5$-adic digits of $x$). Then by induction on $r$ we get
$$
[7x]_{5,\infty}=[7\cdot 5^s y+7z]_{5,\infty}\leq [7\cdot 5^sy]_{5,\infty}+[7z]_{5,\infty}=
$$
$$
=[7y]_{5,\infty}+[7z]_{5,\infty}\leq [y]_{5,\infty} + 2(r-s)+[z]_{5,\infty} +2s=[x]_{5,\infty}+2r.
$$

\emph{Case 2:} All other digits are $=2$, that is, $x=(222...22a)_5$ with $a\in\{0,1,3,4\}$. Note that $7\cdot(222...220)_5=(322...2140)_5$ (where there are two fewer $2$'s on the right hand side). Then
$$
[7x]_{5,\infty}=[7\cdot(22...220)_5 + 7a]_{5,\infty}\leq[7\cdot(22...220)_5]_{5,\infty}+[7a]_{5,\infty}=
$$
$$
=[(322...2140)_5]_{5,\infty}+[7a]_{5,\infty}
=2(r+1)+[7a]_{5,\infty}\leq 2(r+1)+a+2 =
$$
$$
= 2(r-1)+a+6\le 2(r-1)+a+2r = |(222...22a)_5]_{5,\infty}+2r=[x]_{5,\infty}+2r.
$$
\end{proof}

\begin{rmk}Although it will not be used, it follows from the lemma that for every $r \ge 1$ and for every integer $x$ with $0 \le x <5^r$, we have
$$[7x]_{5,\infty}\leq[x]_{5,\infty}+2r+2.$$
Indeed, for $0 \le x <5^r $, the quantity $5x$ is $<5^{r+1}$ and is not $2$ mod $5$. So by the lemma applied to $5x$ with $r+1$, we have
$$[7\cdot 5x]_{5,\infty}\leq[5x]_{5,\infty}+2r+2.$$
But $[7\cdot 5x]_{5,\infty}=[7x]_{5,\infty}$ and $[5x]_{5,\infty}=[x]_{5,\infty}$.
\end{rmk}
\begin{thm}\label{finmono} The geometric monodromy of $\FF_{5,7}$ is finite.
\end{thm}
\begin{proof} By Theorem \ref{numcrit}, we need to show that $[7x]_{5,r}\leq[x]_{5,r}+2r$ for $r\geq 1$ and $0<x<5^r$.
 
If $x=\frac{5^r-1}{2}$, then $x=(22...22)_5$, so $[x]_{5,r}+2r=4r$ and the inequality is clear, since $4r$ is an absolute upper bound for the function $[-]_{5,r}$.

Otherwise, some $5$-adic digit of $x$ is $\neq 2$. Since multiplying $x$ by $5$ cyclically permutes the digits of $x$ modulo $5^r-1$ and does not change the values of $[x]_{5,r}$ or of $[7x]_{5,r}$, we may assume that the last digit of $x$ is $\neq 2$. Then
$$
[7x]_{5,r}\leq [7x]_{5,\infty}\leq[x]_{5,\infty}+2r=[x]_{5,r}+2r
$$
by Lemma \ref{l1}.
\end{proof}

\section{Determination of the monodromy groups}
We first give a general descent construction, valid for general $\sF_{p,D}$ with $D \ge 3$ prime to $p$.
On $\G_m/\F_p$, consider the rank $D-1$ local system $\sH_{p,D}$ whose trace function, for $k/\F_p$ a finite extension, and $t \in \G_m(k)=k^\times$, is
$$t \mapsto -\sum_{x \in k}\psi_k(x^D/t + x).$$
The pullback of  $\sH_{p,D}$ by the $D$'th power map $[D]$ is (the restriction to $\G_m$ of) the local system $\sF_{p,D}$: simply repace $t$ by $t^D$ and make the change of variable $x \mapsto tx$ inside the $\psi$.

View  $\sF_{p,D}$ as the Fourier Transform $FT([D]^\star(\sL_{\psi(x)})$. Then we see from \cite[9.3.2]{Ka-ESDE}, cf. also \cite[2.1 (1)]{Ka-RLSA}, that this $\sH_{p,D}$ is geometrically isomorphic to the Kloosterman sheaf formed with all the nontrivial multiplicative characters of order dividing $D$.

\begin{rmk}\label{traceflddesc}Exactly as in Remark \ref{tracefld}, when $D \ge 3$ is odd and prime to $p$, and $p$ is $1$ mod $4$, the field of traces of $\sH_{p,D}$ lies in $\Q(\mu_p)^+$.
\end{rmk}
This descent has all of its $I(\infty)$-slopes equal to $1/(D-1)$ (either from its identification with a Kloosterman sheaf of rank $D-1$, 
or because its $[D]$-pullback, $\sF_{p,D}$, has all its  $I(\infty)$-slopes equal to $D/(D-1)$).

Either from the fact that its pullback is geometrically irreducible, or from the Kloosterman description, or just from the fact of having all $I(\infty)$-slopes $1/(D-1)$, we see that $\sH_{p,D}$ is geometrically irreducible.

\begin{lem}\label{tensorindec}Let $d \ge 2$, $\ell \neq p$, and $M$ a $d$-dimensional continuous $\overline{\Q_\ell}$-representation $\rho_M$ of $I(\infty)$ all of whose slopes are $1/d$. Suppose that $d$ is not divisible by $p^2$. Then there does not exist a factorization of $d$
as $d=ab$ with $a,b$ both $< d $, together with algebraic groups $G_1 \subset SL(a,\overline{\Q_\ell})$ and $G_2 \subset SL(b,\overline{\Q_\ell})$ such that 
$$Image(\rho_M) \subset {\rm the \ image \ }G_1\otimes G_2 \ {\rm of\ }G_1\times G_2 \ {\rm \ in\ }SL(ab,\overline{\Q_\ell}).$$
\end{lem}
\begin{proof}We argue by contradiction. The map $G_1\times G_2 \rightarrow G_1\otimes G_2$ has finite kernel, $\sK$, which is a subgroup of 
the group $\mu_{\gcd(a,b)}$ (this being the kernel of $SL(a)\times SL(b) \rightarrow SL(ab)$). Because $I(\infty)$ has cohomological dimension one, the group $H^2(I(\infty), \sK)=0$, and therefore there exists a lift of $\rho_M$ to a homomorphism
$$\rho_{a,b}:I(\infty) \rightarrow G_1\times G_2,$$
compare \cite[7.2.5]{Ka-ESDE}.
Because the kernel $\sK$ has order prime to $p$, the upper numbering subgroup $I(\infty)^{\frac{1}{d} +}$, which acts trivially on $M$,
lies in the kernel of $\rho_{a,b}$ (simply because $I(\infty)^{\frac{1}{d} +}$ is a pro-$p$ group which maps to the finite group $\sK$ which has order prime to $p$, cf. \cite[7.1.4]{Ka-ESDE}).
Then the homomorphisms
$$\rho_a:= pr_1  \circ \rho_{a,b}: I(\infty) \rightarrow G_1$$
and 
$$\rho_b:= pr_2 \circ \rho_{a,b}: I(\infty) \rightarrow G_2$$
are  each trivial on  $I(\infty)^{\frac{1}{d} +}$, i.e., each has all slopes $\le 1/d$. Therefore their Swan conductors have $Swan(\rho_a)\le a/d < 1$ and $Swan(\rho_b)\le b/d < 1$. But Swan conductors are nonnegative integers. Therefore both $\rho_a$ and $\rho_b$ have $Swan=0$, i.e., both are tame. But then $M$ is tame, contradiction.
\end{proof}

When $D \ge 3$ is odd and prime to $p$, the half-Tate twist $\sH_{p,D}(1/2)$ is symplectically selfdual.

We now turn our attention to the particular case of $\sH_{5,7}$ and its half-Tate twist $\sH_{5,7}(1/2)$. We know from
Theorem \ref{finmono} that its $G_{geom}$ (and hence also its $G_{arith}$, by Proposition \ref{arithfin}) is a finite irreducible subgroup of $Sp(6, \overline{\Q_\ell})$. By Remark \ref{traceflddesc}, the field of traces of  $\sH_{5,7}(1/2)$ lies in $\Q(\mu_5)^+ =\Q(\sqrt{5})$. Computing the trace at $t=1 \in \F_5^\times$, we see that its field of traces is in fact $\Q(\sqrt{5})$ (and not just $\Q$).

\begin{lem}The group $G_{geom} \subset Sp(6, \overline{\Q_\ell})$ is primitive, i.e.,  the given six-dimensional representation is not induced.
\end{lem}
\begin{proof}By Pink's theorem \cite[Lemma 12]{Ka-MG}, if a Kloosterman sheaf is (geometrically) induced, its list of characters is Kummer induced. So for our Kloosterman sheaf, formed with the nontrivial characters of order $7$, being induced would imply that, for some divisor $n \ge 2$ of $6$, its characters are all the $n$'th roots of some collection of $6/n$ characters. In particular, some ratio of
distinct characters of order $7$ would be a character of order dividing $n$, for some divisor $n$ of $6$, which is not the case: all such ratios have order $7$.

Another proof is to observe that if $\sH_{5,7}$ were induced, then its pullback  $\sF_{5,7}$ would be induced (a system of imprimitivity for a group remains one for any subgroup). But by \cite[11.1]{Such}, if $\sF_{5,7}$ were induced, it would be Artin-Schreier induced, so its rank, $6$, would be a multiple of $p=5$.
\end{proof}

\begin{thm}The local system $\sH_{5,7}(1/2)$ has $G_{geom}=G_{arith} =2.J_2$.
\end{thm}
\begin{proof}
Our situation now is that we have a primitive (by Lemma \ref{tensorindec}) irreducible subgroup $G$ (the $G_{geom}$ for $\sH_{5,7}(1/2)$) in 
$Sp(6, \overline{\Q_\ell})$ such that the given six-dimensional representation is not contained in the tensor product of two lower dimensional representations of $G$. Therefore the larger finite group $G_{arith}$ is a fortiori itself primitive and irreducible inside $Sp(6)$.

We now appeal to the work \cite[\&3, Theorem]{Lind} of Lindsey, as stated in \cite[Theorem 3.1]{C-S}. This gives 
the list of irreducible primitive subgroups of $SL(6,\C)$. Those whose given six-dimensional representation is not contained in a nontrivial tensor product are either
\begin{itemize}
\item[(1)]$2.S_5$ or $S_7$.
\item[(2)]a quasisimple group.
\item[(3)] a group containing a quasisimple group of index two on which the representation remains irreducible.
\end{itemize}

The first case is subsumed by the third, as $2.S_5=2.A_5.2$ contains $2.A_5$, and $S_7$ contains $A_7$.
The quasisimple groups in question are
$$2.A_5 =SL(2,5), 3.A_6, 6.A_6, A_7, 3.A_7, 6.A_7, PSL(2,7), SL(2,7), SL(2,11), $$
and
$$SL(2,13), PSp(4,3) \cong PSU(4,2), SU(3,3), 6.PSU(4,3), 2.J_2, 6.PSL(3,4).$$

Of these, those which lie as an index two subgroup of a larger group inside $SL(6,\C)$ are 
$$SL(2,5), 3.A_6, A_7, PSL(2,7), PSp(4,3),$$
$$SU(3,3),  6.PSU(4,3),6.PSL(3,4).$$

Of the listed quasisimple groups, the only ones with irreducible symplectic representations of degree six are 
$$SL(2,5),\ SL(2,7),\ SL(2,13),\ SU(3,3), \ 2.J_2.$$. 

Of these, the only ones whose field of character values (for any of its six-dimensional irreducible symplectic representations) lies in $\Q(\sqrt{5})$ are $SL(2,5)$,  $SU(3,3)$ and $2.J_2$. For $SL(2,5)$ and $SU(3,3)$, the field of traces is $\Q$; for $2.J_2$ it is  $\Q(\sqrt{5})$.
So the only possibilities for $G_{arith}$ other than $2.J_2$ are the groups $G.2$ for $G$ either $SL(2,5)$ or $SU(3,3)$. But for neither of these two groups does the given representation extend to a symplectic representation (or a selfdual one), as one checks by looking in the Atlas \cite{ATLAS}.

Therefore $G_{arith}$ for $\sH_{5,7}(1/2)$ must be $2.J_2$. As $G_{geom}$ is a normal subgroup of $G_{arith}$ with cyclic quotient (namely some finite quotient of $\Gal(\overline{\F_p}/\F_p)$), we must also have $G_{geom}=G_{arith} =2.J_2$. 
\end{proof}

\begin{cor} For the local system $\sG_{5,7}$, we have $G_{geom}=G_{arith} =2.J_2$.
\end{cor}
\begin{proof}Neither $G_{geom}$ nor $G_{arith}$ changes when we pass from $\A^1$ to the dense open set $\G_m$. Restricted to $\G_m$, 
$\sG_{5,7}$ is the $[7]$ pullback of $\sH_{5,7}(1/2)$. This pullback replaces the  $G_{geom}$ and $G_{arith}$ of $\sH_{5,7}(1/2)$ by normal subgroups of themselves of index dividing $7$. But $2.J_2$ has no such proper subgroups.
\end{proof}

\section{Appendix: Relation of $[x]_{p,r}$ to Kubert's $V$ function}
We denote by $(\Q/\Z)_{{\rm prime \ to\ }p}$ the subgroup of $\Q/\Z$ consisting of those elements whose order is prime to $p$.
We denote by $\Q_p^{n.r.}$ the fraction field of the Witt vectors of $\overline{\F_p}$. For $\F_{q}$ a finite extension of $\F_p$, we have
the Teichmuller character
$$Teich_{\F_q}: \F_{q}^\times \cong \mu_{q-1}( \Q_p^{n.r.}),$$
whose reduction mod $p$ is the identity map on $\F_{q}^\times $.
For an integer $d$, consider the Gauss sum over $\F_q$,
$$G(\psi_{\F_ q},Teich^{-d}) := \sum_{x \in \F_q^\times}\psi_{\F_q}(x)Teich^{-d}(x).$$
If we write $q=p^r$, then by Stickelberger's theorem,
$${\rm ord}_q(G(\psi_{\F_ q},Teich^{-d}) )=(1/r)\sum_{j=0}^{r-1}<p^j \frac{d}{p^r-1}>.$$
As explained in \cite[p. 206]{Ka-G2hyper}, standard properties of Gauss sums show that there is a unique function 
$$V:  (\Q/\Z)_{{\rm prime \ to\ }p} \rightarrow {\rm the\ real\  interval \ }[0,1)$$
such that for $q=p^r$ and $d$ an integer, we have
$$V\left(\frac{d}{p^r-1}\right)=(1/r)\sum_{j=0}^{r-1}<p^j \frac{d}{p^r-1}>.$$
As noted in \cite[line before Theorem 1]{R-L}, we thus have the identity
$$V\left(\frac{d}{p^r-1}\right) = \frac{1}{r(p-1)}[d]_{p,r}$$
{\bf provided} that $1 \le d \le p^r-2$ (i.e., provided that $\frac{d}{p^r-1}$ is nonzero in $(\Q/\Z)_{{\rm prime \ to\ }p}$). However, for $d=0$, $V(\frac{d}{p^r-1})=0$, while 
$$ \frac{1}{r(p-1)}[0]_{p,r}=1.$$

This ``reversal" of the values at $0$, together with the identity for Kubert's $V$ function
$$V(x) + V(-x) =1, \ {\rm \ for \ } x \neq 0,$$
means precisely that for any integer $d$ and any power $p^r$ of $p$, we have the identity
$$\frac{1}{r(p-1)}[d]_{p,r} =1 - V\left(\frac{-d}{p^r-1}\right).$$
With this identity, one sees easily that the criterion \cite[Theorem 1]{R-L} for $\sF_{p,D}$ to have finite geometric monodromy, namely
that
$$[Dx]_{p,r} \le [x]_{p,r} +r(p-1)/2$$
for all $r \ge 1$ and all integers $x$, is equivalent to the criterion \cite[first paragraph after 5.1]{Ka-RLSA}
that for all $y \in (\Q/\Z)_{{\rm prime \ to\ }p}$, we have
$$V(Dy) + 1/2 \ge V(y).$$

\end{document}